\documentclass{article}

\usepackage{amsmath,amsfonts,amsmath,amsthm}
\usepackage{fullpage}
\usepackage{graphicx}
\usepackage{wasysym}
\usepackage{hyperref}
\usepackage{xcolor}

\newcommand{\eps}{\varepsilon}
\newcommand{\Aut}{\mathrm{Aut}}

\newcommand{\R}{\mathbb{R}}

\newcommand{\eee}{\mathrm{e}}

\newcommand{\xx}{\mathbf{x}}
\newcommand{\yy}{\mathbf{y}}

\newcommand{\Exp}{\mathbb{E}}
\newcommand{\Var}{\mathbf{Var}}
\renewcommand{\Pr}{\mathbb{P}}
\newcommand{\Prob}{\Pr}
\newcommand{\dt}{\mathrm{d}t}
\newcommand{\dtau}{\mathrm{d}\tau}

\newcommand{\dz}{\mathrm{d}z}

\newcommand{\ddp}{\delta}

\newcommand{\poly}{\mathrm{poly}}

\newcommand{\perm}{\mathop{\mathrm{Perm}}}
\newcommand{\diamG}{\Delta}

\newtheorem{theorem}{Theorem}
\newtheorem{lemma}{Lemma}

\newtheorem{proposition}{Proposition}
\theoremstyle{definition}
\newtheorem{definition}{Definition}

\begin{document}

\title{
Codes, Lower Bounds, and Phase Transitions in the\\ Symmetric Rendezvous Problem} 
\author{Varsha Dani\thanks{Department of Computer Science, University
    of New Mexico. \texttt{varsha@cs.unm.edu} }  \and
Thomas P. Hayes\thanks{Department of Computer Science, University of New Mexico. \texttt{hayes@cs.unm.edu} } 
\and Cristopher Moore\thanks{Santa Fe Institute. \texttt{moore@santafe.edu}}
\and Alexander Russell\thanks{Computer Science and Engineering,
University of Connecticut.
\texttt{acr@cse.uconn.edu}}
}
\date{}
\maketitle

\thispagestyle{empty}

\begin{abstract}
In the rendezvous problem, two parties with different labelings of the vertices of a 
complete graph are trying to meet at some vertex at the same time. It is well-known 
that if the parties have predetermined roles, then the strategy where one of them 
waits at one vertex, while the other visits all $n$ vertices in random order is optimal, taking at most $n$ steps and averaging about $n/2$.  
Anderson and Weber~\cite{anderson1990rendezvous} considered the 
\emph{symmetric} rendezvous problem, where both parties must use the same randomized strategy.  They analyzed strategies where the parties repeatedly play 
the optimal asymmetric strategy, determining their role independently each time by a 
biased coin-flip. By tuning the bias, Anderson and Weber achieved an expected meeting time of about $0.829 n$, which they conjectured to be asymptotically optimal.

We change perspective slightly: instead of minimizing the
expected meeting time, we seek to maximize the probability of
meeting within a specified time $T$. The Anderson-Weber strategy, which fails with constant probability when $T= \Theta(n)$, is not asymptotically optimal for large $T$ in this setting.  Specifically, we exhibit a symmetric 
strategy that succeeds with probability $1-o(1)$ in $T=4n$ steps.  This is tight: for 
any $\alpha < 4$, any symmetric strategy with $T = \alpha n$ fails with constant probability. Our strategy uses a new combinatorial object that we dub a ``rendezvous code,'' which may be of independent interest.

When $T \le n$, we show that the probability of meeting within $T$ steps is indeed 
asymptotically maximized by the Anderson-Weber strategy.  Our results imply new lower bounds, showing that the
best symmetric strategy takes at least $0.638 n$ steps in expectation.  
We also present some partial results for the symmetric rendezvous problem on other
vertex-transitive graphs.
\end{abstract}

\newpage
\setcounter{page}{1}

\section{Introduction}

The rendezvous problem is a cooperative game first introduced by Alpern in 1976, who also called it the ``Mozart Caf\'e problem''~\cite{alpern-wiley}:
\begin{quote}
Two friends agree to meet for lunch at the Mozart Caf\'e in Vienna on the 1st of January 2000. However, on arriving at Vienna airport, they are told there are three (or $n$) caf\'es with that name, no two close enough to visit in the same day. So, every day each can go to one of them, hoping to find his friend there.  This is a player-symmetric problem, so what is the best randomized strategy for both to adopt?
\end{quote}
While the two players can each label the caf\'es $1,\ldots,n$, they have no way to do this in common.  Thus their labelings 
differ by a permutation $\pi$ chosen uniformly at random from all $n!$ permutations.  How can they ensure a rendezvous as soon as possible?

One approach is for one player to the same location every day, while the other visits them all in arbitrary order---a strategy known as ``Wait For Mommy''~\cite{alpern-gal}, since (in an analogous problem) children are often told that if they can't find their parent, they should stay still while their parent looks for them.  In this case, the players will meet with certainty within $n$ steps, and the expected time to meet is $(n+1)/2$.  It is easy to see that this is optimal; we review here a proof from Anderson and Weber~\cite{anderson1990rendezvous}.  Recall the following lemma:
\begin{lemma} \label{lem:exp-prob}
For any non-negative integer-valued random variable $\tau$, 
\begin{equation} 
\label{eq:Exp-Prob}
\Exp[\tau] = \sum_{t \ge 0} \Pr[\tau > t]\,. 
\end{equation}
\end{lemma}
\noindent
Let $x_t$ and $y_t$ denote the (private) labels chosen by the respective players
at step $t$, and let $\tau$ be the rendezvous time,
that is, $\tau = \min \{t : x_t = \pi(y_t)\}$.  Since $\pi$ is uniformly random, the probability that $x_t = \pi(y_t)$ for any given $t$ is $1/n$.  Therefore, the union bound tells us that 
$\Pr[ \tau \le t ] \le t/n$, and so 
\[ \Pr[\tau > t] \ge (1 - t/n)^+ \, , \]
where $x^+$ denotes $\max(x,0)$.  Thus Lemma~\ref{lem:exp-prob} gives
\begin{equation} \label{eq:(n+1)/2-lb}
\Exp[\tau] \ge \sum_{t = 0}^{n} \left( 1 - \frac{t}{n} \right) = \frac{n+1}{2}.
\end{equation}
The Wait For Mommy strategy matches this general lower bound, and is
therefore optimal.


However, the Wait For Mommy strategy is an asymmetric strategy: it
assumes that the players have agreed on their respective roles in
advance.  In the \emph{symmetric} rendezvous problem, they must
independently follow the same randomized strategy.  This might occur,
for instance, if two players are trying to find a shared radio
frequency on which to communicate (see, e.g.,~\cite{radio}) without
each knowing the other's identity in advance.  

One simple symmetric strategy is for each player to go to a uniformly
random location at each step: then the expected time to meet is $n$,
and the probability that they meet in the first $T$ steps is
$1-(1-1/n)^T \approx 1 - \eee^{-T/n}$.  Another strategy is for each player to visit the $n$ locations in order according to their labeling.  This succeeds in the first $n$ steps as long as the permutation $\pi$ relating their labelings is not a \emph{derangement}, i.e., if it has at least one fixed point $\pi(x)=x$.  A classic exercise in combinatorics shows, using the inclusion-exclusion formula, that the number of derangements is 
\[
n! \left( 1 - 1 + \frac{1}{2} - \frac{1}{6} + \cdots + \frac{(-1)^n}{n!} \right)
\approx \frac{n!}{\eee}\,. 
\]
Since $\pi$ is chosen uniformly at random from among all $n!$ permutations, the probability that they meet in the first $T=n$ steps is again roughly $1-1/\eee$.


We can also think of the rendezvous problem as a cooperative two-player
game called Reduce The Permanent.  There is an $n \times n$ matrix $M$,
whose entries are initially all $1$s.  At each step, one player chooses a row
$x$, the other chooses a column $y$, and they set $M_{xy}=0$, corresponding to
visiting locations $x$ and $y$ according to their respective labelings.  
Recall that the permanent is 
\[
\perm M = \sum_{\pi} \prod_y M_{\pi(y),y}\,, 
\]
where the sum is over all $n!$ permutations.  Thus, at a given point in the game, the permanent of $M$ is the number of
permutations $\pi$ such that $M_{\pi(y),y} = 1$ for all $y$.  This is the number
of permutations $\pi$ such that, if the players' labelings differ by $\pi$, they have not yet met; that is, they have not set any of the entries $M_{xy}$ where $x=\pi(y)$ to zero.  Therefore, the probability they have not yet met is 
\[
\frac{1}{n!} \perm M\,. 
\]
In this setting, Wait For Mommy corresponds to keeping (say) $x$ constant and progressively eliminating a row of the matrix, setting $M_{xy}=0$ for $T$ distinct values of $y$. This reduces $\perm M$ to zero after $T=n$ steps.  On the other hand, if both players wander, then after $T=n$ steps they eliminate the transversal of the matrix corresponding to $\pi$; equivalently, after permuting the rows of the matrix, they eliminate the diagonal.  Then $\perm M$ is the number of derangements of $n$ items, giving $(1/n!) \perm M \approx 1/\eee$.  
In the symmetric setting, the sequences $x_1, \dots, x_t, \dots$ and
$y_1, \dots, y_t, \dots$ must be chosen independently from the same distribution. 
The challenge is then, in expectation, how quickly $\perm M_t$ can be
reduced as a function of $t$, where $M_t$ denotes the all-ones matrix
modified so that all the $(x_i, y_i)$ entries are zeroes, for $1 \le i
\le t$.


Anderson and Weber~\cite{anderson1990rendezvous} considered symmetric strategies where, at the beginning of each block of $n$ moves, each player flips a biased coin, and either waits or wanders for $n$ steps \`a la Wait For Mommy.  Let us add a distinguished label $0$ for a ``frequent'' location which a player might visit repeatedly.  With probability $\theta$ each player waits for $n$ steps at their location $0$, and with probability $1-\theta$ they visit all $n$ locations, ordered by their labeling.  With probability $\theta^2$ they both wait; then they only meet if $\pi(0) = 0$, which occurs with probability $1/n$.  With probability $(1-\theta)^2$ they both wander; then they meet if $\pi$ has a fixed point, i.e., if it is not a derangement.  
Finally, with probability $2 \theta (1-\theta)$ they carry out the Wait For Mommy strategy and meet with certainty.  If they fail to meet, they each update their labeling with a uniformly random permutation---in effect, choosing a new uniformly random $\pi$---and try again for another block of $n$ steps. By optimizing over $\theta$, Anderson and Weber showed that by setting $\theta = 0.249\ldots$ it is possible to meet in expected time $(0.829\ldots) n$.  They conjectured that this is asymptotically optimal, i.e., that the constant $0.829\ldots$ cannot be improved.

In the non-asymptotic regime, Weber~\cite{weber-three-locations,
  weber-four-locations} has recently proved that the Anderson-Weber
strategy is optimal when $n=3$, but \emph{not} when $n=4$.  

\paragraph{Our contributions.}  
Rather than minimizing the expected time to meet, we focus on the time
$T$ it takes to meet with high probability.  The Anderson-Weber
strategy fails with probability $\Omega(1)$ for any constant number of
rounds, i.e.,  for every finite $\alpha$, we have $\Prob(\tau > \alpha n) = \Omega(1)$.
It is tempting to think that this is true for every symmetric strategy.  
Surprisingly, this is false!
Indeed, we show that $\Prob (\tau > \alpha n)$ can be made to tend to 
zero as $n \to \infty$ if and only if $\alpha \ge 4$.
Specifically, we give a strategy for which $\Prob(\tau > 4n) = O(1/\log n)$, 
using a new combinatorial object we call a ``rendezvous code.''  

On the other hand, we show that if our goal is to maximize the 
probability of meeting within the first $T$ steps, the Anderson-Weber 
strategy is indeed asymptotically optimal in the case $T \le n$, in 
the sense that it achieves the greatest possible success probability,
up to an additive $o(1)$ error.  Note that the optimal bias $\theta$ depends 
on $T$.  We also show that any symmetric strategy requires, in 
expectation, at least $T=0.638 n$ steps.  Our results suggest that 
for $n < T < 4n$, the optimal strategy may be an interesting
interpolation between the Anderson-Weber strategy and our rendezvous codes.

Finally, we consider the symmetric rendezvous problem on graphs $G$
other than the complete graph~\cite{alpern-gal}.  Now the players are only allowed to move along
$G$'s edges at each step; on the other hand, the permutation $\pi$ is
limited to the automorphism group $\Aut(G)$.  To keep the problem
interesting, we focus on vertex-transitive graphs, so that there is no
distinguished vertex where they can arrange to meet.  We show that if
$G$ is the cycle, where their labelings differ by a random translation
and (with probability $1/2$) a flip between clockwise and
counterclockwise, the probability of failure is $\Omega(1)$ for any
$T=O(n)$.  On the other hand, we show that on a cycle of prime
length they can meet with high probability in $T=n$ steps if we
decorate the cycle with some additional edges.  
We also show how our rendezvous codes can be adapted to arbitrary graphs of
diameter $o(n)$, giving a symmetric strategy which guarantees a
failure rate of $o(1)$ when $T \ge 8n$.
We close with open questions and directions for further work.

\section{Notation and Preliminaries}


It will be helpful to adopt formal definitions for some terminology
that will be used throughout the paper.  First, we define the space
from which each player must make a selection; 
we call its elements ``walk schedules.''
\begin{definition} A \emph{walk schedule} is a
(finite or countably infinite) sequence of labels from $\{0, 1,
\ldots, n\}$.  We view label 0 as an alternative notation for label
$n$.  This is for notational convenience, as shall become apparent later. 
A player, having selected a walk schedule $\xx = (x_1, x_2, \dots)$ in the
rendezvous game, visits, at each step $t$, the vertex labeled $x_t$
under her own private vertex labeling.
\end{definition}
\noindent In section~\ref{sec:other-graphs}, we will generalize the notion of
walk schedule to graphs other than the complete graph.

We think of a game in which two players each choose a walk schedule independently of each other.  A third party, known as the Adversary, chooses a permutation $\pi$ of the $n+1$ vertices, independently and without any knowledge of the walk schedules. 
\begin{definition} The \emph{rendezvous time}, $\tau$ is defined as
  $\tau = \min\{t : x_t = \pi(y_t)\}$.  
\end{definition}
\noindent
The two players collaborate to try to minimize $\tau$, while the
Adversary tries to maximize $\tau$.  Thinking of the two players as a 
single team, we have the following notions of what their strategies
can be.
\begin{definition} A \emph{(mixed) strategy} is a distribution over pairs of walk schedules.  A strategy is \emph{symmetric} if it is of the form $\mathcal{D} \times \mathcal{D}$; that is, each player independently samples their walk schedule from the same distribution $\mathcal{D}$.
\end{definition}
\noindent
As for the Adversary, her optimal strategy is to choose a 
uniformly random permutation.  To see this, observe that 
either or both of the players can randomly permute the labels
used in their own walk schedule, which has the same effect.
For this reason, we will henceforth assume $\pi$ is chosen uniformly
at random, and forget about the Adversary.

Our next definition allows us to analyze arbitrary strategies in a ``Wait For Mommy''-like way, 
classifying steps as ``waiting'' if a player visits a small number of locations many times, 
and ``wandering'' if she visits a large number of places a few times each.

\begin{definition} \label{def:freq-rare} 
Let $\xx$ be a walk schedule.  If a label $\ell$ appears more 
than $n^{2/3}$ times in $\xx$, we call $\ell$ \emph{frequent} for $\xx$;  
otherwise we call $\ell$ \emph{rare}.  We refer to $t$ as a \emph{waiting step} for $\xx$ 
if $x_t$ is a frequent label, and a \emph{wandering step} if $x_t$ is rare; 
more colloquially, we say that ``$\xx$ waits'' or ``$\xx$ wanders'' on step $t$.  
\end{definition}

At this point, two comments should be made.  First, when we talk about
symmetric strategies, the walk schedules should always be of infinite
length.  This is because, for any finite $T$, there is a nonzero chance that both
players will choose the same label sequence for the first $T$ steps,
in which case there is at least a $1/\eee$ conditional probability
that $\tau > T$.  

On the other hand, we shall be considering probabilities of the form
$\Prob( \tau \ge \alpha n)$, for which we may as well assume the walk
sequences are truncated to length $\alpha n$.  It is in the latter
context that we will be applying the terminology from Definition~\ref{def:freq-rare}.
Note that, in this case, counting shows that there can be at most
$\alpha n^{1/3}$ rare labels per walk sequence.

\begin{definition}
We say that a strategy is of \emph{Anderson-Weber type}, if for each $j= 0, 1, 2, \dots$, with probability one, 
the walk sequence in the block of steps $\{jn + 1, jn + 2, \ldots, (j+1)n\}$ is 
either $0^n$ or a permutation of $1, \dots, n$.
\end{definition}

\section{The (Almost) Complete Story on the Complete Graph}

In this section we present our main results on the symmetric rendezvous problem for the complete graph.  Specifically, we show that using ``rendezvous codes'' rather than Anderson-Weber strategies, it is possible to rendezvous with high probability in $4n$ steps; but that for $T = \alpha n$ with $\alpha < 4$, any symmetric strategy fails with probability $\Omega(1)$.  
This result came as a big surprise to us:  we were initially convinced that the optimal strategy 
was of Anderson-Weber type, or at least that it consists of independent rounds of length $n$, 
but clearly for some $T < 4n$ this ceases to be the case.  On the other hand, we also show that for $T \le n$, the probability of success is asymptotically maximized by Anderson-Weber strategies.

\subsection{A bound based on waiting and wandering}

We begin with an intuitive yet technical lemma which is crucial to the rest of the paper.   In addition to motivating our definition of rendezvous codes, presented in Section~\ref{sec:rdvcode}, it also underlies all of our lower bounds.

\begin{lemma}  \label{lem:abcd-formula}
Let $\xx$ and $\yy$ be two walk schedules of length $T = O(n)$. Let
$a,b,c,d$ be as follows:
\begin{itemize}
\item[$\bullet$] $a$ is the number of steps when both $\xx$ and $\yy$ wait,
\item[$\bullet$] $b$ is the number of steps when $\xx$ waits and $\yy$ wanders, 
\item[$\bullet$] $c$ is the number of steps when $\xx$ wanders and $\yy$ waits, and 
\item[$\bullet$] $d$ is the number of steps when both $\xx$ and $\yy$ wander.
\end{itemize}
Note that $a+b+c+d = T$.
Then the probability that two players playing according to schedules $\xx$
and $\yy$ fail to rendezvous by time $T$ is bounded below by 
\begin{equation} \label{eq:abcd}
\Pr[\tau > T] 
\ge (1-b/n)^+ (1-c/n)^+ \eee^{-d/n} - o(1)\,.
\end{equation}
\end{lemma}

Before we present the proof, we briefly describe some intuition behind it.  
On steps when both players wait, there is a $o(1)$ chance of
meeting because each player has $O(n^{1/3})$ frequent labels.
On steps when one player waits and the other wanders, they are, at best,
executing the Wait For Mommy strategy.  This gives rise to the terms $(1-b/n)^{+}$ and $(1-c/n)^{+}$.  
On steps when both players wander, they meet on each step with
probability $1/n$.  If these events were independent, then the
probability of failing to meet would be
$(1-1/n)^d \approx \eee^{-d/n}$.  If these four ``phases'' were run
independently, we would obtain the right-hand side of \eqref{eq:abcd}.
We remark that, from the above observations, it is not hard to see
that the right-hand side of \eqref{eq:abcd} is tight up to the
additive $o(1)$ term for all values of $a,b,c,d$.

Our proof will rely on the following result about the permanent of a
zero-one matrix with relatively few zeroes in every row and column.

\begin{lemma} \label{lem:perm-exponential}
Let $M$ be an $m \times m$ matrix of zeroes and ones, and let $Z \subset [m] \times [m]$
be the index set of its zeroes.   Suppose $M$ has no more than $\kappa$
zeroes in any row or column, and that $\kappa = o(m)$.  Then, for a
randomly chosen permutation $\pi$,
\[
\Prob( \forall (x,y) \in Z,  \pi(y) \ne x ) = 
\frac{\perm(M)}{m!} = \eee^{-|Z|/m} (1 + o(1))\,.
\]
\end{lemma}

\begin{proof}
First, observe that if $|Z| = o(m)$, then the union bound implies
\[
\Prob( \forall (x,y) \in Z, \pi(y) \ne x ) \ge 1 - \frac{|Z|}{m} = 1-o(1)\,,
\]
so since $\eee^{-|Z|/m}$ is also $1 - o(1)$, we are done.  Henceforth,
assume $|Z| = \Omega(m)$, so in particular, $\kappa = o(|Z|)$.

For each $S \subseteq Z$, define 
an event $A_S = \{ \forall \mathbf{p} = (x,y) \in S: \pi(y) = x\}$, where $\pi$
is a uniformly random permutation of $[m]$.  Since $\perm(M)$ counts
the permutations that miss all zeroes of $M$, we have
\[
\frac{\perm(M)}{m!} = \Prob\left( \overline{\bigcup_{\mathbf{p} \in Z}
  A_{\{\mathbf{p}\}}} \right).
\]
Bonferroni's inequalities give us upper and lower bounds on this
quantity in terms of the first few terms of the inclusion-exclusion
formula:
\[
\Prob\left( \overline{\bigcup_{\mathbf{p} \in Z}
  A_{\{\mathbf{p}\}}} \right) 
= \sum_{i=0}^{j} (-1)^i \!\!\!\!\!\! \sum_{S \subset Z: |S| = i}
\Prob( A_S) \alpha_i\,,
\]
where $\alpha_i = 1$ for $0 \le i \le j-1$, and $0 \le \alpha_j \le 1$.
We will choose the number of terms, $j = j(m,|Z|, k)$, so that $j \to
\infty$ slowly enough that $j^2 = o(m)$ and $j^2 \kappa = o(|Z|)$.
  
Now, $\Prob(A_S)$ equals zero if some two elements of $S$ lie in the
same row or column; otherwise, it equals $(m-|S|)!/m!$.
There are $\binom{|Z|}{i}$ ways to choose $S \subset Z$ such that
$|S| = i$.  However, at most
\[
|Z| (\kappa-1) \binom{|Z|-2}{i-2}
\]
of these
contain two elements from the same row or column: the pair from the 
same row or column can be chosen in at most $2 |Z| (\kappa-1) / 2$ ways,
since each row or column has at most $\kappa$ elements from $Z$, and the
pair can be chosen in either order.
This equals 
\[ \frac{(\kappa-1)i(i-1)}{|Z|-1} \binom{|Z|}{i} = o\left(\binom{|Z|}{i}\right) \, \] 
since $i \le j$ and we chose $j$ such that $j^2 \kappa = o(|Z|)$.  It follows that
\begin{align*}
\frac{\perm(M)}{m!}  &= \sum_{i=0}^{j} \frac{(-1)^i}{i!}  \frac{\binom{|Z|}{i}}{\binom{m}{i}} (1
+ o(1)) \alpha_i \\
&= \sum_{i=0}^{j} \frac{(-1)^i}{i!}  \left(\frac{|Z|}{m}\right)^i (1
+ o(1)) \alpha_i \\
&= \exp(-|Z|/m) (1 + o(1))\,,
\end{align*}
where the last step follows by the absolute convergence of the
Maclaurin series for $\exp(-|Z|/m)$ (since $j \to
\infty$), and the penultimate step follows by standard approximations
to binomial coefficients, using our assumption that $j^2 = o(m)$ and
$j^2 = o(|Z|)$.  This completes the proof.
\end{proof}

We now proceed to the proof that this approximately maximizes the rendezvous probability.

\begin{proof} [Proof of Lemma~\ref{lem:abcd-formula}]
We will prove a more detailed estimate, which characterizes $\Prob(\tau > T)$ up to an additive $o(1)$, and from which the desired lower bound, \eqref{eq:abcd}, follows easily.  Let $k$ be the number of $\xx$'s frequent labels, and for $1 \le i \le k$ let $b_i$ be the number of distinct rare labels in $\yy$ on steps when $\xx$ waits at its $i$'th frequent label.  Then $b \ge \sum_{i=1}^k b_i$.
Similarly, let $\ell$ be the number of $\yy$'s frequent labels; then $c \ge \sum_{j=1}^\ell c_j$ where $c_j$ is the number of distinct rare labels in $\xx$ on steps when $\yy$ waits at its $j$'th frequent label.
Finally, let $d' \le d$ be the number of distinct pairs $(x_t, y_t)$ of rare labels from steps when both $\xx$ and $\yy$ wander.  We will show that the probability that $\xx$ and $\yy$ fail to meet is 
\begin{equation} \label{eq:detailed}
\Pr[\tau > T] = 
\prod_{i=1}^k \left( 1 - \frac{b_i}{n} \right) \;
\prod_{j=1}^\ell \left( 1 - \frac{c_j}{n} \right) \;
\eee^{-d'/n} + o(1)\,. 
\end{equation}
This expression is minimized subject to the constraints $0 \le b_i, c_j \le n$, $b \ge \sum_{i=1}^k b_i$, $c \ge \sum_{j=1}^\ell c_j$, and $d' \le d$ when $k=\ell=1$, $b_1 = \min(b, n)$, $c_1 = \min(c, n)$, and $d' = d$, from which \eqref{eq:abcd} follows immediately.  Thus all that remains is to prove~\eqref{eq:detailed}.

To see this, let us condition on the portion of $\pi$ that matches up frequent labels.  More precisely, we reveal the values of $\pi(y)$ for all $y$ that are frequent labels for $\yy$, and of $\pi^{-1}(x)$ for all $x$ that is a frequent label for $\xx$.  Note that by the definition of ``frequent'' (Definition~\ref{def:freq-rare}) we have $k, \ell \le n^{1/3}$.  It follows that, with probability $1 - k \ell / n = 1 - O(n^{-1/3})$, there are no pairs $(x,y)$ where $\pi(y) = x$ and $x$ is frequent for $\xx$ and $y$ is frequent for $\yy$.

Now, our conditioning has revealed $\pi(y)$ for exactly $k + \ell$ labels $y$.  Note that, had we sampled, fully independently, a uniformly random label from $\xx$ (resp.\ $\yy$) to pair with each frequent label from $\yy$ (resp.\ $\xx$), we would, with high probability, had no conflicts, and conditioned on this, sampled from the same distribution.  But, for the fully independent sampling, the probability of failing to rendezvous at one of the locations with a frequent label is exactly
\[
\prod_{i=1}^k \left( 1 - \frac{b_i}{n} \right) \;
\prod_{j=1}^\ell \left( 1 - \frac{c_j}{n} \right)\,.
\]

Conditioned on the pairings for the frequent labels on both sides, 
there remain $n-(k+\ell)$ unpaired rare labels on each side, and
$\pi$ is a uniformly random pairing of these.  So the conditional
probability of failing to rendezvous at one of these times equals
\[
\frac{\perm M}{(n-k-\ell)!}\,,
\]
where $M$ is the zero-one matrix whose entries are indexed by the remaining 
$n-(k+\ell)$ unpaired vertices on each side, and whose 
$(x,y)$ entry is 0 if there exists a step $t$ for which $x_t = x$ and $y_t = y$.

We note that $M$ has at most $n^{2/3}$ zeroes in any single row or column since its rows and columns are indexed by rare labels.
Let $d''$ be the total number of  zeroes  in $M$. Thus there are $d'' \le d'$ distinct pairs of labels $(x_t, y_t)$  such that $x_t$ is rare for $\xx$, $y_t$ is rare for $\yy$, and neither $x_t$ nor $y_t$ is matched to a frequent label by $\pi$.  By linearity of expectation, the expected number of steps $t$ such that $\pi(y_t)$ is a frequent label in $\xx$ equals $T k/n = O(n^{1/3})$, and similarly for the number of steps $t$ such that $\pi^{-1}(x_t)$ is a frequent label in $\yy$.  Hence, with probability $1 - o(1)$, we have $d'' \ge d' - O(n^{1/3})$.  

Applying Lemma~\ref{lem:perm-exponential} to $M$, we see that 
\begin{align*}
\frac{\perm M}{(n-k-\ell)!} &= \eee^{-d''/(n-k-\ell)} (1 + o(1))\\
&= \eee^{-d'/n} (1+o(1))\,.
\end{align*}
Putting it all together gives the desired formula~\eqref{eq:detailed}.
\end{proof}

\subsection{Rendezvous codes}
\label{sec:rdvcode}

Lemma~\ref{lem:abcd-formula} tells us that, in order to guarantee a
rendezvous by time $T$ with high probability, we need to ensure
that at least one of $b$ or $c$ is $n - o(n)$ with high probability.
In other words, our symmetric strategy must have Wait For Mommy hidden
inside it with high probability.  This motivates the following definition.

\begin{definition} A \emph{rendezvous code} is a collection $R$ of walk
  schedules such that, for every two distinct schedules
  $\xx, \yy \in R$, there is a set of $n$ steps in which 
  $\xx$ has every label exactly once, while $\yy$
  repeats the same label $n$ times.  
\end{definition}

If $R$ is a rendezvous code of length $T$ and size $s$, then 
the symmetric strategy in which both players independently sample 
their walk schedule from $R$ will ensure a meeting within $T$ steps 
whenever their schedules are distinct, which occurs with probability at least $1-1/s$.  
In fact, it ensures at least two meetings with probability at least $1 - 1/s - 1/n$, 
since if $\xx$ and $\yy$ are distinct, each player will 
wait while the other one wanders through all $n$ locations.  
(The $1/n$ term corresponds to the possibility that the players' repeated labels 
correspond to the same vertex.)

It is not obvious that rendezvous codes exist of length $O(n)$ and size $\omega(1)$.  
In this section, we present 
rendezvous codes of length $T = 4n$ and size $\Omega(\log n)$, thus
ensuring a rendezvous in at most $4n$ steps with probability
$1-O(1/\log n)$.   
As a result, Anderson-Weber-type strategies do not maximize the
probability of success for $T \ge 4n$.  As we will show below, 
the constant $4$ is tight: the shortest possible rendezvous codes 
of size $\omega(1)$, and more generally, the shortest possible 
symmetric strategies that succeed with probability $1-o(1)$, have length $4n - o(n)$.


Our construction of rendezvous codes of length $O(n)$ and size $\Omega(\log n)$ is as follows.  
There will be a unique frequent label, $0$, that occurs $T/2 = 2n$ times in every walk
schedule.  The remaining labels, $1, \dots, n$, are rare, with each 
occuring exactly twice in every walk schedule.  
Let $n = 2^d$ and $T = 4n = 2^{d+2}$, and consider the $(d+2) \times T$ binary matrix $M$ whose $t$'th column is the binary representation of $t$ for $0 \le t < T$.  Specifically, let $M_{i,t}$ be the $i$th most significant bit of the binary representation of $t$.  For $n=4$, in which case $d=2$ and $T=16$, this matrix is 
\begin{align*}
M &= \left[ \begin{array}{c c c c c c c c c c c c c c c c}
0 & 0 & 0 & 0 & 0 & 0 & 0 & 0 & 1 & 1 & 1 & 1 & 1 & 1 & 1 & 1 \\
0 & 0 & 0 & 0 & 1 & 1 & 1 & 1 & 0 & 0 & 0 & 0 & 1 & 1 & 1 & 1 \\
0 & 0 & 1 & 1 & 0 & 0 & 1 & 1 & 0 & 0 & 1 & 1 & 0 & 0 & 1 & 1 \\
0 & 1 & 0 & 1 & 0 & 1 & 0 & 1 & 0 & 1 & 0 & 1 & 0 & 1 & 0 & 1 
\end{array} \right]\,. 
\end{align*}

We will use $M$ to construct another matrix $M'$ such that each row of $M'$ corresponds to a walk schedule.  
If $M_{i,t} = 0$ then $M'_{i,t} = 0$, indicating that the player waits at location $0$ on those steps. 
Each of the rare locations $1,\ldots,n=2^d$  will be visited twice, as follows.  
For each row $i$, the $T/2 = 2^{d+1}$ columns $t$ for which $M_{i,t} = 1$ 
occur in complementary pairs $t, t'$, in the sense that $M_{j,t} = \overline{M_{j,t'}}$ for all $j \ne i$ (where, as usual, 
$\overline{0} = 1$ and $\overline{1} = 0$).  
We assign each such pair a distinct label $\ell \in [ 2^d]$, so that $M'_{i,t} = M'_{i,t'} = \ell$.  
The precise labeling does not matter, but we could define $\ell$ as follows: if we delete the $i$th row, 
then each column $t$ represents an integer $w_t < 2^{d+1}$ in binary, where $w_t + w_{t'} = 2^{d+1}-1$.  
Then we set $\ell$ equal to $w_t+1$ or $w_{t'}+1$, whichever is smaller.  
Again taking $n=4$, the resulting matrix $M'$ is
\begin{align*}
M' &= \left[
\begin{array}{c c c c c c c c c c c c c c c c}
0 & 0 & 0 & 0 & 0 & 0 & 0 & 0 & 1 & 2 & 3 & 4 & 4 & 3 & 2 & 1\\
0 & 0 & 0 & 0 & 1 & 2 & 3 & 4 & 0 & 0 & 0 & 0 & 4 & 3 & 2 & 1\\
0 & 0 & 1 & 2 & 0 & 0 & 3 & 4 & 0 & 0 & 4 & 3 & 0 & 0 & 2 & 1 \\
0 & 1 & 0 & 2 & 0 & 3 & 0 & 4 & 0 & 4 & 0 & 3 & 0 & 2 & 0 & 1 
\end{array} \right]\,. 
\end{align*}

The rows of $M'$ constitute a rendezvous code of length $T$ and size $s = d+2$.  To see this, let $\xx$ and $\yy$ be the $i$th and $j$th row respectively, where $i \ne j$.  There are $T/4 = n$ columns $t$ where $M_{i,t} = 1$ and $M_{j,t} = 0$.  After removing the $i$th row, the columns $t$ where $M_{i,t} = 1$ range over all $2^{d+1}$ binary sequences, so the corresponding labels $x_t = M'_{i,t}$ range over all $\ell \in  [n]$.  Moreover, since each pair of columns with the same label is complementary, exactly one of them occurs when $y_t = M_{j,t} = 0$.  Thus for each $1 \le \ell \le n$ there is a $t$ where $x_t = M'_{i,t} = \ell$ and $y_t = M'_{j,t} = 0$, so this set of $T/4=n$ columns satisfies the definition of the rendezvous code.  



Now suppose the two players independently choose a walk schedule uniformly from the rows of $M'$, and that these rows are distinct.  Using $0$ to denote their frequent locations, in the unlikely event that $\pi(0) = 0$ they meet on the very first step.  Otherwise, they meet at the time $t$ such that $y_t = 0$ and $x_t = \pi(0)$, and again (if they choose to continue) at the $t$ where $x_t = 0$ and $y_t = \pi^{-1}(0)$.  Thus this strategy only fails in the event that $i=j$ and $\pi$ is a derangement, and this occurs with probability $O(1/s) = O(1/\log n)$.  


There is an easy generalization of this construction to the case where $n$ is not a power of $2$.  Suppose $n \le n' = k 2^d$.  We construct the $(d+2) \times 2^{d+2}$ matrix $M$ as before, but now repeat each row $k$ times.  In the first copy, we set the nonzero entries of $M'$ so that they range from $1$ to $2^d$ as before; in the second copy, we add $2^d$ to these entries so that they range from $2^d+1$ to $2 \cdot 2^d$, and so on.  If we replace entries of $M'$ greater than $n$ with $0$, this produces a rendezvous code of size $s=d+2$ and length $2^{d+2} k = 4n'$.  By setting, say, $d = (1/2) \lceil \log_2 n \rceil$ and $k = \lfloor (n-1)/2^d \rfloor$, we obtain $s = \Omega(\log n)$ and $n' = n+o(n)$, so $T = 4n + o(n)$.

Thus we have proved 
\begin{theorem} \label{thm:code-exists}
For any $n$, there is a symmetric rendezvous strategy on the complete
graph $K_n$ for which $\Prob(\tau \le 4n) \ge 1 - O(1/\log n)$.
\end{theorem}


\subsection{More rendezvous codes}

It is clear from the definition of rendezvous codes that we can form them by concatenation.  That is, given two rendezvous codes $R_1, R_2$, the code $R_1 R_2 = \{ r_1 r_2 : r_1 \in R_1, r_2 \in R_2 \}$ is a rendezvous code as well.  Concatenating the rendezvous code given in the previous section with itself $c$ times gives a rendezvous code of length $4cn$ and size $O(\log^c n)$.  

Of course, this simply corresponds to performing our previous strategy $c$ times independently, reducing the probability of failure to $O(1/\log^c n)$.  In analogy with error-correcting codes, we suspect that there are much better constructions.  One interesting question is whether there is a rendezvous code of length $O(n)$ and polynomial size, allowing the players to meet with probability $1-1/\poly(n)$ in linear time.



In an effort to explore possible constructions of rendezvous codes a
little further, we can construct $M$ and $M'$ using some base other
than binary.  As we will see, this gives rendezvous codes that achieve
any rational fraction $p$ of time spent waiting versus wandering, and
also somewhat reduces the expected time to meet.  In this section, 
which is not essential to the rest of the paper, we sketch out the
main ideas of this construction; details are left to the
interested reader.

Fix constants $A$ and $B$ with $0 < A < B$.  Set $n=(B-A) B^{k-1}$, and let 
\[
T = B^{k+1} = \frac{B^2}{B-A} \,n\,. 
\]
Let $M$ be the $(k+1) \times T$ matrix whose columns are the base-$B$ representations of the numbers $0 \le t < T$. As in the binary case, we will use $M$ to construct a matrix $M'$ whose rows are walk schedules.  
However, now $B-A$ of the possible entries of $M$ will correspond to wandering.  Specifically, an entry $m \in \{ 1, \ldots, B-A \}$ corresponds to wandering among the $B^{k-1}$ locations $(m-1) B^{k-1} + 1, (m-1) B^{k-1} + 2, \ldots, m B^{k-1}$. The
remaining $A$ possible entries of $M$ correspond to waiting at the frequent location, which we again denote $0$; thus if $M_{i,t} \notin \{ 1, \ldots, B-A \}$ we set $M'_{i,t}=0$.  The fraction of time spent waiting is then $p=A/B$.  

It remains to define a location $M'_{i,t}$ in the case $M_{i,t} = m \in \{ 1, \ldots, B-A \}$.  
Fix a row $i$, and consider the $B^k$ columns $t$ such that $M_{i,t} =
m$.  If we delete the $i$th row, the remaining $k$ rows represent, in
base $B$, all integers $0 \le w < B^k$; let $w_t$ denote the integer
thus represented by the $t$th column.  Analogous to the complementary
pairs in the binary construction, we partition these columns into
$B^{k-1}$ equivalence classes of $B$ columns each, where $t$ and $t'$
are equivalent if, for some $z \in \{0,\ldots,B-1\}$, we have $M_{j,t}
\equiv M_{j,t'}+z \bmod B$ for all $j \ne i$.  Within each equivalence class, there is a unique $t'$ such that $0 \le w_{t'} < B^{k-1}$.  Set $\ell_t = w_{t'}+1$ for all $t \equiv t'$, and finally set $M'_{i,t} = (m-1) B^{k-1} + \ell_t$.  

To see this is a rendezvous code, let $1 \le x \le n$, and write $x = (m-1) B^{k-1} + \ell$.  Let $i$ and $j$ be distinct, and let $x_t = M'_{i,t}$ and $y_t = M'_{j,t}$ respectively.  In the equivalence class where $M_{i,t} = m$ and $\ell_t = \ell$, there are $A$ steps $t$ where $M_{j,t} \notin \{1, \ldots, B-A\}$, so that $M'_{j,t}=0$.  On any of these steps, we have $x_t = x$ and $y_t = 0$.  

Varying the constants $A, B$ gives a family of codes of length $T = O(n)$ and size $s = k+1 = \Omega(\log n)$.  While the length $T$ is minimized by the binary case $B=2$, it turns out that we can decrease the \emph{expected} rendezvous time by taking $B > 2$, as we will now show.

It is convenient for our analysis to further modify the code by permuting its $T$ columns randomly.  In that case, we obtain bounds that are similar to Lemma~\ref{lem:abcd-formula}, but which are essentially tight.  

Let $x_t = M'_{i,t}$ and $y_t = M'_{j,t}$ where $i$ and $j$ are distinct.  Consider the first $\tau$ steps of the strategy; we will compute the probability that the players fail to meet during this time.  With high probability, there are roughly $(A/B)^2 \tau$ of these steps when both players wait; if we condition on the likely event that $\pi(0) \ne 0$, these steps are useless to us.  Similarly, with high probability there are roughly $(1-A/B)^2 \tau$ steps when both players wander; on each of these steps, the players meet with probability $1/n$.  These events are roughly independent, so the probability that we fail to meet on any of them is essentially 
\begin{equation}
\label{eq:wanderfail}
(1-1/n)^{(1-A/B)^2 \tau} 
\approx \eee^{-(1-A/B)^2 \tau/n} 
= \eee^{-(B-A) \tau/T}\,. 
\end{equation}
Now, fix one of the ``waiting'' symbols $m \notin \{1, \ldots, B-A\}$.  With high probability there are roughly $((B-A)/B^2) \tau$ steps where $M_{j,t} = m$, so that $\yy$ waits at $y_t=0$, while $M_{i,t} \in \{1, \ldots, B-A\}$ so that $\xx$ wanders.  However, the locations $x_t$ that $\xx$ visits during these steps are all distinct, so the probability that none of them are $\pi(0)$ is 
\[
1-\frac{B-A}{B^2} \frac{\tau}{n} = 1-\frac{\tau}{T}\,. 
\]
There are $A$ waiting symbols, and the players can also rendezvous by having $\xx$ wait while $\yy$ wanders.  Taking the product of all of these along with~\eqref{eq:wanderfail} gives
\[
\Pr[\mbox{no rendezvous in the first $\tau$ steps}] 
\approx \left( 1-\frac{\tau}{T} \right)^{2A} \eee^{-(B-A) \tau/T}\,. 
\]
Integrating this gives us the expected time to rendezvous:
\begin{align*}
\Exp[\tau] 
&= \sum_{\tau=0}^{T} \Pr[\mbox{no rendezvous in the first $\tau$ steps}] \\
&\approx \int_0^T \left(1- \frac{\tau}{T}\right)^{2A} \eee^{-(B-A) \tau/T} \dtau \\
&= \frac{T}{\eee^{B-A}} \int_0^1 z^{2A} \,\eee^{(B-A)z} \dz\,. 
\end{align*}

When $A=1$ this integral can be easily evaluated, which yields
\[
\Exp[\tau] 
= \frac{B^2-4B+5 - 2 \eee^{-(B-1)}}{(B-1)^3} \,T
= \frac{B^2(B^2-4B+5 - 2 \eee^{-(B-1)})}{(B-1)^4} \,n\,. 
\]
Setting $B=2$ and $A=1$ gives the special case of the binary code described above, giving an expected rendezvous time of 
\[
\Exp[\tau] = 4 \left( 1-\frac{2}{\eee} \right) n \approx 1.057 n \,. 
\]
Using base $B=4$ instead, again with $A=1$, reduces the expected time to 
\[
\Exp[\tau] = \frac{16}{81} \left(5-\frac{2}{\eee^3}\right) \approx 0.968 n \,.
\]
We believe this expected time is the best this family of codes can achieve; however, it is larger than the expected time of the Anderson-Weber strategy, and the length $T=(16/3)n$ is greater than $T=4n$ for the binary code.


\subsection{A phase transition at $T=4n$}

We will now show a lower bound on the failure probability for an
arbitrary symmetric strategy when $T = (4-\delta)n$ for $\delta > 0$.
As a corollary, any rendezvous code of size $\omega(1)$ 
must have length at least $4n-o(n)$.  
First we deal with the simpler case of strategies for which 
every walk schedule has the same number of ``waiting'' steps.

\begin{lemma} \label{lem:nobins}  Suppose $\mathcal{D}$ is any
  distribution over walk schedules, and $T \ge 0$.
Then, the symmetric strategy that samples from $\mathcal{D}
  \times \mathcal{D}$ fails to rendezvous within $T$ steps with
  probability at least
\begin{equation} \label{eq:1-T/2n}
\Pr( \tau > T ) 
 \ge \eee^{-T/n} \left(1- \frac{T}{2n}\right) - o(1)\,.
\end{equation}
If we further assume that every walk schedule in $\mathcal{D}$
has equally many waiting steps within the first $T$ steps, then 
we get the stronger lower bound
\begin{equation} \label{eq:1-T/4n}
\Pr( \tau > T ) 
\ge \eee^{-T/n} \left(\left(1- \frac{T}{4n}\right)^{\!+} \right)^{\!2}
- o(1)\,.
\end{equation}
\end{lemma}

\begin{proof} 
Let $a,b,c,d$ be defined as in the
statement of Lemma~\ref{lem:abcd-formula}. 
Here $a,b,c,d$ depend on the particular pair of walk
schedules chosen, and are thus random variables.
Hence Lemma~\ref{lem:abcd-formula}
gives us the lower bound
\begin{equation}
  \label{eq:common-step}
  \begin{aligned}
    \Pr(\tau > T) &= 
    \Exp[\Pr(\tau > T \mid a,b,c,d)] && \mbox{by the Tower Law for
      conditional expectations,}\\
    &\ge 
    \Exp\!\left[\left(1 - \frac{b}{n} \right)^{\!+}  \left(1 - \frac{c}{n}\right)^{\!+}  
      \eee^{-d/n} 
    \right] - o(1) &&\mbox{by Lemma~\ref{lem:abcd-formula} applied to the
      conditional probability,} \\
    &\ge  
    \eee^{-T/n} \,\Exp\!\left[\left(1 - \frac{b}{n} \right)^{\!+}  \left(1
        - \frac{c}{n} \right)^{\!+}  \right] - o(1)
    &&\mbox{since $d \le T$.} 
\end{aligned}
\end{equation}

Note that, for each $1 \le t \le T$, the probability that $\xx$ waits and
$\yy$ wanders at step $t$ is of the form $w_t(1-w_t)$, where $w_t \in
[0,1]$, and hence is at most $1/4$.  It follows by linearity of
expectation that $\Exp[b] = \Exp[c] \le T/4$.  Now, subject only to this
constraint, and to $b, c \ge 0$, and assuming $T \le 2n$, 
an easy shifting argument shows that the expectation in
the right-hand side of \eqref{eq:common-step} is minimized by the 
distribution
\[
(b,c) = \begin{cases}
(n,0) & \mbox{with probability $T/4n$,} \\
(0,n) & \mbox{with probability $T/4n$,} \\
(0,0) & \mbox{with probability $1 - T/2n$,} \\
\end{cases} 
\]
which immediately implies \eqref{eq:1-T/2n}.

In the case when all walk sequences in the support of $\mathcal{D}$
have equally many waiting steps, this number equals $a+b$ and $a+c$ by
definition, and hence $b=c$ with probability one.   Hence
\eqref{eq:common-step} becomes
\[
\Prob(\tau > T) \ge 
\eee^{-T/n} \,\Exp\!\left[ \left(\left(1 - \frac{b}{n}\right)^{\!+}
  \right)^{\!2\,}\right] - o(1)\,.
\]
Now, since the function $f(x) = (x^+)^2$ is convex, so is $g(b) = ((1 - b/n)^+)^2$.  Applying Jensen's inequality to the right-hand side, and again applying $\Exp[b] \le T/4$, yields the desired inequality~\eqref{eq:1-T/4n}.
\end{proof}

Note that, of the bounds in Lemma~\ref{lem:nobins}, \eqref{eq:1-T/2n} is general but becomes trivial for $T \ge 2n$, whereas~\eqref{eq:1-T/4n} is nontrivial for $T < 4n$ but assumes that all walk schedules have the same number of waiting steps.  Our next result relies on a simple binning argument to get a nontrivial lower bound for arbitrary symmetric strategies and all $T < 4n$.  We have made no attempt to optimize the constants in the lower bound.

\begin{theorem} 
\label{thm:4n-lb}
For every symmetric strategy on the complete graph $K_n$ of length $T = (4 - \delta)n$ where $\delta > 0$, 
the probability of failure is bounded below by
\[
\Pr( \tau > T ) \ge \frac{\eee^{-4}}{4096} \,\eee^{\delta/2}
\,\delta^4  -o(1)\,.
\]
\end{theorem}

\begin{proof} 
Let $\eps > 0$ be a constant which we will determine below. 
By the pigeonhole principle, there must exist some $w$ such that with probability
at least $\eps$, a walk schedule sampled from $\mathcal{D}$ waits for
between $w T$ and $(w+\eps)T$ steps.
Thus, with probability at least $\eps^2$, both players will wait for a number of steps in this range.  
We will condition on this event happening.

By adding the correct number of waiting steps at the end of each
of these walk schedules, we can define a new distribution $\mathcal{D}'$
over walk schedules of length $(1+\eps)T$ 
in which every walk waits for exactly $T' = (w + \eps)T$ steps,
and in which the first $T$ steps agree with the original
distribution, conditioned on waiting for between $w T$
and $(w + \eps)T$ steps.

Now set $\eps = \delta/8$, and note that
\[
T' = (1+\eps)T \le (1+\delta/8) (4-\delta) n \le (4-\delta/2)n\,.
\]
Applying Lemma~\ref{lem:nobins} to this new distribution, we find
 \begin{align*}
 \Pr[\mbox{no rendezvous using $\mathcal{D}$}]
 &\ge \eps^2 \,\Pr[\mbox{no rendezvous using $\mathcal{D}'$}] \\
 &\ge \eps^2 \,\eee^{-T'/n} \left(1- \frac{T'}{4n}\right)^{\!2} -o(1) \\
 &\ge (\delta/8)^4 \,\eee^{-(4-\delta/2)}  - o(1)\,,
 \end{align*}
 completing the proof.
 \end{proof}

\subsection{Anderson-Weber strategies are asymptotically optimal for \texorpdfstring{$T \le n$}{T <= n}}

In this section, we show that strategies of Anderson-Weber type asymptotically 
maximize the probability of meeting within a fixed deadline $T \le n$,
up to an additive $o(1)$ error.  Of course, 
with such a short time horizon, there is only time for the first of the
wait-or-wander coin flips, so it is a very special case.
On the other hand, these early steps are the biggest contributors
to the sum in~\eqref{eq:Exp-Prob}, so their importance should not be
underestimated. 

\begin{theorem} 
\label{thm:AW-for-T<n} 
For each $T \le n$, $\Pr[\tau \ge T]$ is asymptotically minimized, up to an additive $o(1)$ error, 
over all symmetric strategies, by some strategy of Anderson-Weber type.
\end{theorem}


Our proof will use some tools from the minimization of submodular functions on lattices. 

\begin{definition}
Let $S$ be a finite set of size $n$, and let $f : 2^S \to \R$ be a
real-valued function on subsets of $S$.
We say $f$ is \emph{submodular} if, for every $X, Y \subseteq S$, 
we have
\[
f(X \cap Y) + f(X \cup Y) \le f(X) + f(Y)\,.
\]
\end{definition}

We begin by recalling some basic tools for constructing submodular
functions.  The first of these is a 
standard connection between submodularity and concavity\footnote{A function $f$  defined on a discrete set $S$ is concave if and only if it can be extended to a concave function on the convex hull of $S$.}   (e.g.~\cite[Prop. 5.1]{lovasz1983submodular}):
\begin{proposition} \label{prop:submod-concave}
Suppose $S$ is a finite set, and $f : 2^S \to \R$ is defined by $f(Z)
= g(|Z|)$ for some $g: \{0, \dots, |S|\} \to \R$.  Then $f$ is
submodular if and only if $g$ is concave.
\end{proposition} 

\noindent 
We will also use the following easy consequence of the distributive law:
\begin{proposition} \label{prob:submod-intersection}
If $f(Z)$ is submodular then, for any fixed $X$, $f(X \cap Y)$ is a submodular function of $Y$.
\end{proposition}

\noindent Finally, we will use the following fact about products of submodular functions (e.g.~\cite[Ex.~3.32(b)]{boyd2004convex}).
\begin{proposition} 
\label{prop:submod-product}
Let $g$ and $h$ be non-negative, submodular functions on the same domain. Suppose $g$ is non-decreasing with respect to set inclusion, and $h$ is non-increasing. Then their product $gh$ is submodular.
\end{proposition}

Using these, we can prove that the function featured in
Lemma~\ref{lem:abcd-formula} is submodular.  More precisely,

\begin{proposition} 
\label{prop:lb-is-submodular} 
Let $T \le n$ and let $X, Y \subseteq [T]$.
Let $a = |\overline{X} \cap \overline{Y}|$, $b = |\overline{X} \cap Y|$, $c = |X \cap \overline{Y}|$, and $d = |X \cap Y|$.  Denote 
\[
f(X,Y) = (1 - b/n) (1-c/n) \,\eee^{-d/n}.
\]
Then $f$ is submodular in each of its arguments.
\end{proposition}
\noindent
We remark that the assumption $T \le n$ cannot be eliminated.
Also note that, when $\xx, \yy$ are walk schedules of length 
$T \le n$ and $X,Y$ are the respective sets of wandering steps,
then $a,b,c,d$ are the same as in Lemma~\ref{lem:abcd-formula}
and the conclusion of Lemma~\ref{lem:abcd-formula} is that
\[
\Prob( \tau > T ) \ge f(X,Y) - o(1)\,.
\]


\begin{proof}[Proof of Proposition~\ref{prop:lb-is-submodular}]
Since $f$ is symmetric, it suffices to prove that, for each fixed $X$, $f(X,Y)$  is submodular in $Y$.  To see this, observe that 
\[
f(X,Y) = 
\left( 1 - \frac{|\overline{X} \cap Y|}{n} \right)
\left( 1 - \frac{|X \cap \overline{Y}|}{n} \right)
\exp\!\left( - \frac{|X \cap Y|}{n} \right)
= g(\overline{X} \cap Y) \,h(X \cap Y) \, , 
\]
where 
\[
g(Z) = 1 - \frac{|Z|}{n} 
\quad \text{and} \quad 
h(Z) = \left(1 - \frac{|X| - |Z|}{n} \right) \exp\!\left( -\frac{|Z|}{n} \right)\,. 
\]
Clearly $g(Z)$ is a concave, non-increasing function of $|Z|$.  Using a little calculus and taking into account the fact that $|Z| \le |X|$, we see that $h(Z)$ is a concave, non-decreasing function of $|Z|$.  By Proposition~\ref{prop:submod-concave}, both are then submodular functions of $Z$, and Proposition~\ref{prob:submod-intersection} implies that $g(\overline{X} \cap Y)$ and $h(X \cap Y)$ are submodular functions of $Y$.  Finally, $g(\overline{X} \cap Y)$ is non-increasing, $h(X \cap Y)$ is non-decreasing, and both are non-negative, so Proposition~\ref{prop:submod-product} implies that $f(X,Y)$ is submodular in $Y$.
\end{proof}

\begin{definition}
Let $x \in [0,1]^S$.  We say that a distribution $\phi$ over subsets
of $S$ has expectation $x$ if, for all $s \in S$, 
$\Pr_{S' \sim \phi}[s \in S'] = x(s)$.

The \emph{Lov\'asz extension of $x$} is the distribution over $2^S$
that samples $\lambda \in [0,1]$ uniformly at random, and outputs the subset 
$\{s \in S \colon x(s) \ge \lambda \}$.  Note that this distribution has expectation $x$.  Moreover, 
its support forms a chain $S_0 \subseteq \cdots \subseteq
S_k$; for instance, if the $x(s)$ are distinct and less than $1$, 
we have $k=n$, 
$S_0 = \emptyset$, and $S_{i} = S_{i-1} \cup \{s_i\}$ where $x(s_1) < x(s_2) < \cdots < x(s_n)$.  
\end{definition}

The Lov\'asz extension plays an important role in the minimization of submodular functions 
because of the following property (see, for instance, \cite[Theorem 3.7]{dughmi2009submodular}):

\begin{proposition} \label{prop:lovasz-min-basic}
Let $f: 2^S \to \R$ be submodular, and let $x \in [0,1]^S$.  
Then, among all distributions whose expectation is $x$,  
the Lov\'asz extension of $x$ minimizes the expected value of $f$.
\end{proposition}

\noindent 
This means that, for a fixed expectation $x \in [0,1]^S$, \emph{every}
submodular function is minimized by the same distribution, whose support 
forms a chain.  

Rather than directly applying Proposition~\ref{prop:lovasz-min-basic},
we will need the following 
``two-player'' analog of it.

\begin{proposition} 
\label{prop:lovasz-min-advanced} 
Let $S$ be a finite set, and let $f: 2^S \times 2^S \to \R$ be submodular in each of its arguments. 
Let $x \in [0,1]^S$, and let $\psi$ be the 
Lov\'asz extension of $x$.  Then the distribution $\psi \times \psi$
minimizes the expectation of $f$ among all 
distributions $\phi \times \phi$ such that $\phi$
has expectation $x$.
\end{proposition}

\begin{proof}
The proof is a straightforward ``uncrossing'' argument along the lines of \cite[Lemma 3.6]{dughmi2009submodular}.  Since $S$ is finite, the space of all probability distributions over $2^S$ is compact.  Hence, by the Extreme Value Theorem, there exists a distribution $\phi$ that, among all distributions with expectation $x$, minimizes $\Exp_{\phi \times \phi}(f)$, and, subject to both these constraints, maximizes $\Var_{X \sim \phi}(|X|)$.  We will prove that this $\phi$ must, in fact, be the Lov\'asz extension of $x$.

Suppose $X, Y \subseteq S$ were two incomparable subsets in the support of $\phi$.  
Then, taking $\ddp = \min\{\phi(X), \phi(Y)\}$, 
we can define a modified distribution $\phi'$ 
by decreasing $\phi(X)$ and $\phi(Y)$ by $\ddp$, and increasing
$\phi(X \cup Y)$ and $\phi(X \cap Y)$ by $\ddp$; $\phi'$ is otherwise equal
to $\phi$.
This preserves the expectation $x$.  Moreover, it preserves or decreases the expectation of $f$, since
\begin{align*}
\Exp_{{\phi' \times \phi'}}[f] 
- \Exp_{{\phi \times \phi}}[f] 
&= \left( \Exp_{{\phi' \times \phi}}[f] 
- \Exp_{{\phi \times \phi}}[f] \right) 
+ \left( \Exp_{{\phi' \times \phi'}}[f] 
- \Exp_{{\phi' \times \phi}}[f] \right) \\
&= \ddp \,\Exp_{Z \sim \phi} [ f(X \cup Y, Z) + f(X \cap Y, Z) - f(X,Z) - f(Y,Z) ] \\
&+ \ddp \,\Exp_{Z \sim \phi'} [ f(Z, X \cup Y) + f(Z, X \cap Y) - f(Z,X) - f(Z,X) ] \\
&\le 0 \, ,
\end{align*}
where the last line follows since $f$ is symmetric and submodular in both arguments.
Finally, we observe that 
\[
\Var_{Z \sim \phi'} |Z| = \Var_{Z \sim \phi} |Z| + 2 \delta \; |X
\setminus Y| \; |Y
\setminus X| > \Var_{Z \sim \phi} |Z|\,.
\]
But this contradicts our choice of
$\phi$.  Hence our assumption that $\phi$ contained
incomparable elements was false, so the support of $\phi$ is a chain.
Since the expectation of $\phi$ is $x$, it must be the Lov\'asz extension of $x$.
\end{proof}

As a corollary, we will deduce that the expectation of a concave
function of two variables is minimized for two-point distributions:
\begin{proposition}
\label{prop:shifting}
Let $g(t_1,t_2)$ be concave in each of its arguments, and suppose that $t_1$ and $t_2$ are chosen independently from a distribution $P$ supported on $\{0,\dots,T\}$. Then, among all distributions $P$ with a given expectation $\mu = \Exp_P[t]$, the expectation $\Exp_{P \times P}[g]$ is minimized by a distribution $\overline{P}$ supported only at the extremes of the interval: $\overline{P}(T) = \mu/T$, $\overline{P}(0) = 1-\mu/T$, and $\overline{P}(t) = 0$ for all other $t$.
\end{proposition}

\begin{proof}

Let $S = [T]$ and define a function $f: 2^S \times 2^S \to \R$ by $f(X,Y) = g(|X|,|Y|)$.
By Proposition~\ref{prop:submod-concave}, $f$ is submodular in each of
its arguments.  

Let $P$ be a distribution on $\{0, \dots, T\}$ of mean $\mu$.
We define a distribution $\phi$ on $2^S$ by coupling it with $P$.
For each $0 \le t \le T$, whenever $t$ is sampled from $P$, 
let $\phi$ produce a uniformly random subset of $S$ of size $t$.
Note that our definition ensures that $\Exp_{\phi \times \phi} (f) =
\Exp_{P \times P}(g)$.

Define $x \in [0,1]^S$ by $x_i = \mu/T$ for every $1 \le i \le T$.
Observe that $x$ equals the expectation of $\phi$.
By Proposition~\ref{prop:lovasz-min-advanced},
$f$ is minimized at $\psi \times \psi$, where $\psi$ is 
the Lov\'asz extension of $x$.  However, the Lov\'asz extension
of $x$ is supported on just two points: $\emptyset$ with probability
$1 -\mu/T$, and $[T]$ with probability $\mu/T$.
It follows that
\begin{align*}
\Exp_{\overline{P} \times \overline{P}} (g) &=
\Exp_{\psi \times \psi} (f) && \mbox{by inspection,} \\
&\le \Exp_{\phi \times \phi}(f) && \mbox{by Proposition~\ref{prop:lovasz-min-advanced},} \\
&= \Exp_{P \times P}(g) && \mbox{by the definition of $\phi$},
\end{align*}
which was to be proved.
\end{proof}


Our next proposition narrows the optimal strategies down further to those that wait for some number of steps and wander for the remaining steps:

\begin{proposition}
\label{prop:rv}
For each $T \le n$, it is possible to define a probability distribution $P(t)$ on $[0,T]$ such that
the following symmetric strategy maximizes the probability of
a successful rendezvous, up to an additive $o(1)$: each player
independently samples a value of $t$ according to $P(t)$, waits at location $0$ on steps
$\{1, 2, \ldots, t\}$ and wanders through $T-t$ distinct locations on 
steps $\{t+1, \ldots, T\}$.
\end{proposition}

\begin{proof}
Let $T \le n$ and let $\mathcal{D} \times \mathcal{D}$ be a symmetric strategy.  
As before, we consider walk sequences $\xx, \yy$ sampled from
$\mathcal{D} \times \mathcal{D}$, and let $X, Y \subseteq [T]$ denote
the corresponding set of wandering steps.  Also as before, let
$a,b,c,d$ be as defined in Lemma 2, and let
\[
f(X,Y) = \left(1 - \frac{b}{n}\right)^{+}  \left(1 - \frac{c}{n}\right)^{+}  \eee^{-d/n} 
= \left(1 - \frac{b}{n}\right) \left(1 - \frac{c}{n}\right)
\eee^{-d/n}\,,
\]
where the second equality uses the hypothesis that $T \le n$.
Note that, also because $T \le n$, by Proposition~\ref{prop:lb-is-submodular}, $f$ is a submodular function.
Let $\phi$ denote the distribution over $2^{[T]}$ from which $X$ is
sampled, let $x$ denote the expectation of $\phi$, and let $\psi$
denote the Lov\'asz extension of $x$.  Then we have
\begin{align*}
\Prob( \tau > T ) & \ge \Exp_{(X,Y) \sim \phi \times \phi} [ f(X,Y) ] - o(1) && \mbox{by Lemma~\ref{lem:abcd-formula},} \\
& \ge \Exp_{(X,Y) \sim \psi \times \psi} [ f(X,Y) ] - o(1) && \mbox{by Proposition~\ref{prop:lovasz-min-advanced}.} 
\end{align*}

Now, since the support of $\psi$ is a chain, and we can permute the
steps $\{1, \dots, T\}$ without changing the success probability, we
can, without loss of generality, push all the waiting steps to the front.  
This means that, under $\psi$, we can assume that the set, $X$, of wandering steps
is always of the form $\{t+1, \dots, T\}$.  

Considering, for instance, \eqref{eq:detailed}, we see that, up to an
additive $o(1)$, it is optimal to choose all distinct labels for the
wandering steps, which completes the proof.
%
\end{proof}

We are now ready to prove the main result of the section.

\begin{proof}[Proof of Theorem~\ref{thm:AW-for-T<n}]
Suppose $\mathcal{D}$ is an asymptotically optimal strategy of the type described in Proposition~\ref{prop:rv}, and let $P$ be the distribution described in Proposition~\ref{prop:rv}.

We condition on the waiting times $t_1, t_2$ sampled independently
from $P$ by the two players.
In the notation of Lemma~\ref{lem:abcd-formula}, we have
$b=(t_1-t_2)^+$, $c=(t_2 - t_1)^+$, 
and $d = T-\max(t_1,t_2)$.  Thus the probability of failing to meet is $g(t_1,t_2)+o(1)$, where 
\[
g(t_1,t_2) = \left(1 - \frac{|t_1 - t_2|}{n} \right) \eee^{(\max(t_1, t_2) - T)/n}\,. 
\]
It is easily verified that $g(t_1,t_2)$ is concave in each of $t_1$ and $t_2$ separately.  
To see this, first note that $g$ is continuous: then, fixing $t_1$, the second derivative of 
$g$ with respect to $t_2$ is zero for $t_2 < t_1$ and negative for $t_2 > t_1$.   
Thus $g$ is concave as a function of $t_2$, and by symmetry as a function of $t_1$ as well.

Up to $o(1)$, the probability of failure is the expectation of
$g(t_1,t_2)$ where $t_1, t_2$ are chosen independently according to
$P$.  Since $g$ is concave in both of its arguments,
Proposition~\ref{prop:shifting} shows that, for any fixed $\mu =
\Exp[t]$, the expectation of $g$ is minimized when $P$ is supported on
$\{0,T\}$.  That is, setting $\theta = \mu/T$, with probability $P(T) = \theta$ we wait for $T$ steps, and with probability $P(0) = 1-\theta$ we wander for $T$ steps.  Thus the strategy is of Anderson-Weber type, completing the proof. 
\end{proof}

\subsection{Bounding failure and expected time in symmetric strategies}

We conclude our discussion of the complete graph by using the results of the previous section to bound 
the probability of failure, and the expected time to meet, for any symmetric strategy.

First, now that we know the asymptotically optimal strategy for $T \le n$ is of Anderson-Weber type, what is the optimal value of the waiting probability $\theta$?  Using Lemma~\ref{lem:abcd-formula}, the probability of failure in the first $T \le n$ steps is 
\begin{equation}
\label{eq:no-rend-theta}
\Pr[\mbox{no rendezvous by time $T$}] 
= \theta^2 + 2 \theta (1-\theta) (1-T/n) + (1-\theta)^2 \,\eee^{-T/n} - o(1)\,. 
\end{equation}
For a given $T$, this is minimized by
\begin{equation}
\label{eq:theta-min}
\theta 
= \frac{1 - (1 - T/n) \,\eee^{T/n}}{1 - (1 - 2T/n) \,\eee^{T/n}}\,.
\end{equation}
In particular, for $T=n$ this recovers the bias $1/(\eee+1) = 0.269\ldots$ of the original Anderson-Weber strategy.  
Plugging this $\theta$ back into~\eqref{eq:no-rend-theta} gives
\begin{equation}
\label{eq:no-rend-min}
\Pr[\tau > T]
= \frac{1 - (1 - T/n)^2 \,\eee^{T/n}}{1 - (1 - 2T/n) \,\eee^{T/n}} - o(1)\,. 
\end{equation}
Thus, the right-hand side of \eqref{eq:no-rend-min} is a lower bound for the failure probability of any symmetric strategy.
Plugging into~\eqref{eq:Exp-Prob} then gives the following lower bound on the expected time to meet, 
\begin{align*}
\frac{\Exp[\tau]}{n} 
&= \frac{1}{n} \sum_{T=0}^\infty \Pr[\tau > T] \\
&\ge \frac{1}{n} \sum_{T=0}^n \Pr[\tau > T] \\
&\ge \frac{1}{n} \sum_{T=0}^n \frac{1 - (1 - T/n)^2 \,\eee^{T/n}}{1 - (1 - 2T/n) \,\eee^{T/n}} - o(1)
 \\
& = \int_0^1 \frac{1 - (1 - t)^2 \,\eee^{t}}{1 - (1 - 2t) \,\eee^{t}} \,\dt - o(1) \\
&= 0.6027\ldots - o(1)\,. 
\end{align*}

We can improve this bound a little as follows.  Note that $\Pr[\tau > n] \ge 1/(\eee+1)$.
Since the probability of meeting, not necessarily for the first time,
in any particular step is $1/n$, a union bound tells us that the
probability of meeting in any of steps $n+1, \dots, T$ is at most $(T-n)/n$.
Hence, for $T > n$ we have
\[
\Pr[\tau > T] 
\ge \frac{1}{\eee+1} - \frac{T-n}{n} 
= \frac{\eee+2}{\eee+1} - \frac{T}{n} \, , 
\]
which is nontrivial for $n < T < (\eee+2) n / (\eee + 1)$.
Therefore, 
\begin{align*}
\frac{\Exp[\tau]}{n}
&\ge \frac{1}{n} \sum_{T=0}^{\lfloor n(\eee+2)/(\eee+1) \rfloor} \Pr[\tau > T] \\
&\ge \int_0^1 \frac{1 - (1 - t)^2 \,\eee^{t}}{1 - (1 - 2t) \,\eee^{t}} \,\dt 
+ \int_1^{(\eee+2)/(\eee+1)} \left( \frac{\eee+2}{\eee+1} - t \right) \dt - o(1) \\[1ex]
&= 0.6389\ldots - o(1)\,.
\end{align*}
This implies
\begin{proposition} 
\label{prop:exp}
For every symmetric strategy, the expected time to the first rendezvous is bounded by 
\[
\Exp[\tau] \ge 0.6389 n - o(n)\,. 
\]
\end{proposition} 

This bound is not tight: for instance, we can get a slight further
improvement by applying the bound of Theorem~\ref{thm:4n-lb} for
values of $T$ up to $4n$.  More importantly, our bound is derived by 
analyzing each summand $\Pr(\tau > T)$ in the sum~\eqref{eq:Exp-Prob}
separately, even though we know that for the optimal Anderson-Weber
strategies, these correspond to different values of $\theta$.

In particular, it is still possible that the expected time $(0.829\ldots) n$
achieved by the Anderson-Weber strategy is asymptotically optimal.  
Nevertheless, since the asymmetric Wait For Mommy strategy achieves 
$\Exp[\tau] = n/2 - o(n)$, Proposition~\ref{prop:exp} shows that the 
first rendezvous takes $\Theta(n)$ longer in expectation when we are 
restricted to symmetric strategies.

\section{Rendezvous on Other Graphs} \label{sec:other-graphs}

In this section we consider the rendezvous problem on general graphs $G$.  
Each player has a map of $G$ with vertices labeled $1,\ldots,n$, and
their labelings differ by a permutation $\pi$
chosen uniformly at random from the automorphism group $\Aut(G)$.  

Unless $G$ is the complete graph (or the empty graph) $\Aut(G)$ is a proper subgroup of the permutation group 
on the $n$ vertices. In particular, not all pairs of vertices can be mapped onto each other, i.e., $\Aut(G)$ is not 2-transitive.  
This restriction on the permutation $\pi$ can help the players meet each other; on the other hand, 
the players can only move along $G$'s edges, which hinders them.  (We could consider a different model, where the 
players can see $G$'s edges but move to any vertex at each step; we will not do this here.)

If there is a vertex which is topologically unique---say, with the highest degree, or at the center of a 
star---the players can simply arrange to meet there.  More generally, if $\Aut(G)$ is not vertex-transitive, then the 
players might as well focus their attention on the smallest equivalence class of vertices.  Thus we will assume that 
$\Aut(G)$ is vertex-transitive, so that all vertices are topologically equivalent.


We start by stating some simple upper bounds:
\begin{lemma}
\label{lem:ham}
Let $G$ be any connected graph on $n$ vertices, with diameter $\diamG = o(n)$. Then there is a symmetric strategy that ensures a rendezvous with probability $1-o(1)$ in $8n+o(n)$ steps. If, additionally, $G$ is Hamiltonian then the strategy takes only $4n+o(n)$ steps.
\end{lemma}

\begin{proof}
The proof consists of simulating the strategy of Theorem~\ref{thm:code-exists} for the complete graph.
We will first discuss the Hamiltonian case. 

Without loss of generality, each player's labeling has the property that the vertices in sequence, 
$1, \dots, n$, form a Hamiltonian path.  As before, we will also use $0$ to denote the ``frequent'' or ``waiting'' location.  
Let $M$ and $M'$ be the matrices described in Section~\ref{sec:rdvcode}.  Recall that these matrices have width $T= 4n$; moreover, the sequence of rare locations follows the path $1,\ldots,n$ forward and then backward, with waiting steps to location $0$ interspersed.  

Now, for any $k \le \log n$, let $M_k$ and $M'_k$ be the top $k$ rows of $M$ and $M'$ respectively. We note that each row of  $M_k$ can be divided into exactly $2^k$ blocks of equal length, where within each block, the row contains all zeroes or all ones. Correspondingly, the rows of $M'_k$ have the property that within each block, the strategy either waits at location $0$ for the entire block or visits the next $n/2^k$ locations in the (forward or backward) Hamiltonian path. We now modify $M'_k$ to have $\diamG$ extra columns between each pair of consecutive blocks. This allows the walker to get from their position at the end of each block to their desired position at the beginning of the next block, even if there is no edge between the corresponding vertices. This adds $2^k \diamG$ columns to $M'_k$, increasing the length $T$, and gives a strategy that fails with probability $1/k$. If we set $k = (1/2) \log (n/\diamG)$, the length is increased by $\sqrt{n \diamG} = o(n)$, and the strategy succeeds with probability $1-1/k = 1-o(1)$.

The only thing keeping the same approach from working when $G$ is not Hamiltonian is the absence of a path of length $n +o(n)$ that visits every vertex. Suppose each player chooses a spanning tree of $G$. The depth-first traversal of the spanning tree is a closed walk of length $2n$ which visits each location at least once. Thus, starting with a rendezvous code for $2n$ vertices, and mapping the locations $1, 2, \dots, 2n$ to the locations in the order they appear in the traversal achieves the desired end.
\end{proof}

It is worth noting that the ``Lov\'asz conjecture,'' open since at least 1969 (see~\cite{pak2009hamiltonian}) asserts that all vertex-transitive graphs are Hamiltonian.  In order to get our upper bound of $4n$ above, we only need a much weaker property, namely that the graph contain a walk of length $n + o(n)$ that reaches all $n$ vertices (and which is allowed to repeat vertices).

In the remainder of this section, we will give examples of graphs where a rendezvous is considerably easier, or considerably harder, to arrange than on the complete graph.

\subsection{Graphs with rendezvous in $n$ steps}

It turns out that the lower bound of $4n$ that applies to the complete graph (Theorem~\ref{thm:4n-lb}) does not hold for general graphs.  Here we give an example of a vertex-transitive graph where the players can meet with high probability in just $n$ steps.  

Let $n$ be prime, let $k < \lfloor n/2 \rfloor$, and let $G$ be the circulant graph whose vertices are $0, 1, \dots, n-1$ where for each $i$ there is an edge between $i$ and $i \pm j \bmod n$ for  $1 \le j \le k$.   Note that each vertex has degree $2k$.  The automorphism group $\Aut(G)$ is the dihedral group $D_{n}$: that is, the group of rigid rotations that send $x$ to $x+s \bmod n$ for some $s$, and ``flips'' that send $x$ to $-x+s \bmod n$ for some $s$.

Now consider the following symmetric strategy.  Each player picks a uniformly random ``velocity'' $j \in \{-k, -(k-1), \ldots, k-1, k\}$.  Starting at their initial vertex, which each one calls $0$, they move $j$ steps around the cycle at each step, visiting location $jt \bmod n$ at time $t$.  Let $j_1, j_2$ be the two players' velocities; we can assume without loss of generality that the two players' labelings differ by a rotation $\pi$ rather than a flip, since otherwise we can replace $j_2$ with $-j_2$.  

If $j_1 \ne j_2$, then the players will meet with certainty within $n$ steps.  To see this, let $x_2-x_1 \bmod n$ be their initial displacement from each other on the cycle.  Since $j_1-j_2$ is invertible mod $n$, there is a $t < n$ such that $x_1 + j_1 t = x_2 + j_2 t \bmod n$.  This fails only if $j_1=j_2$, which occurs with probability $1/(2k+1)$.  Thus if we let $k = k(n)$ grow as any increasing function of $n$, this strategy succeeds with high probability within $n$ steps.


All that is needed in this construction is that $j_1-j_2$ has a high probability to be invertible in $\mathbb{Z}_n$.  Thus it also works when $n$ is a product of fairly large primes, for instance if all prime factors are $\omega(\log n)$ while $k = O(\log n)$. 


\subsection{A lower bound for the cycle}

Next we consider the $n$-cycle, without any additional edges (for which the automorphism group is still the dihedral group).  
In this case, we will show that, unlike the complete graph or the circulant graph of the preceding section, there is no symmetric strategy with $T=O(n)$ that succeeds with high probability.  

For this graph, there is a very real possibility of the two players moving in opposite directions along the same edge
at the same time.  When this occurs, we will generously treat them as having rendezvoused, even though they did not 
reach the same vertex at the same time.  Our lower bound on this model holds in the original model as well.

First consider the following simple algorithm, which Anderson and Weber~\cite{anderson1990rendezvous} 
credit to Steve Alpern.  Starting at a random location, pick a random direction and walk for $n/2$ steps. If you
haven't met by that time, pick a new random direction and repeat.  
If the players choose the same direction, the distance between them
remains constant, and they do not meet.  However, as soon as they
choose opposite directions, their paths must cross and they will meet
(invoking our generous convention about meeting on an edge).  

It is easy to see that this algorithm gives an expected meeting
time of (asymptotically) $3n/4$.  In expectation, the players spend one round 
moving in the same direction for $n/2$ steps; then, in the first round
that they move in opposite directions, they meet in expected time $n/4$.  
Indeed, according to~\cite{anderson1990rendezvous}, Alpern conjectured that this strategy 
asymptotically minimizes the expected meeting time on the cycle.  

However, for any integer $k$, the probability of not having
met by time $kn/2$ is $2^{-k}$.  To put it differently, for any $T$, 
this strategy fails with probability $2^{-O(T/n)}$, 
so $T=\omega(n)$ steps are needed in order to meet with probability $1- o(1)$.  
We will now show that this is unavoidable.

\begin{theorem}
Let $n$ be a multiple of $6$.  For any symmetric strategy for the rendezvous problem on
the $n$-cycle, and for any $T$, 
\[
\Pr[\mbox{no rendezvous in the first $T$ steps}] 
\ge \frac{1}{2} \,3^{-6T/n}\,. 
\]
\end{theorem}

\begin{proof}
A symmetric strategy $\gamma$ that runs in time  $T = \tau n$ is a distribution on walks of length $T$ on $C_{n}$.  Such a walk may be described as a sequence of vertices $(v_0, v_1, \ldots, v_T) \in [n]^{T+1}$ where $v_0$ is the starting vertex and for all $t$, $v_{t+1}$ is either the same as $v_t$ or adjacent to it, i.e.,
\begin{equation}
\label{eq:cyc-adj}
v_{t+1} \in \{v_t -1, v_t, v_t + 1\}\,.
\end{equation}
(mod $n$, of course).  Equivalently, the walk may be described in terms of its starting location and its \emph{moves},  
\[
(v_0, (s_1, s_2, \dots, s_T))\in [n] \times \{-1,0,1\}^T \quad \mbox{where}\quad s_t = v_t - v_{t-1} \, ,
\]
so $\gamma$ induces a distribution on $[n] \times \{-1, 0, 1\}^T$. 

Let $m=n/6$, and define a mapping from walks of length $T$ on $C_{n}$ to walks of length $\tau =T/m$ on $C_6$ by dividing the cycle into $6$ sectors of size $m$, and stroboscopically observing where they are when $t$ is a multiple of $m$:
\[
(v_0, v_1, \ldots, v_T) 
\mapsto \left(\left\lfloor \frac{v_0}{m}\right\rfloor , 
\left\lfloor \frac{v_m}{m}\right\rfloor, 
\left\lfloor \frac{v_{2m}}{m}\right\rfloor, \ldots, 
\left\lfloor \frac{v_T}{m}\right\rfloor \right)\,.
\]
It is easy to see that the expression on the right is a walk on $C_6$: 
by~\eqref{eq:cyc-adj} and the triangle inequality,
$|\lfloor v_{i+m}/m \rfloor -\lfloor v_i/m \rfloor| \le 1$.
Thus $\gamma$ induces a distribution on $[6] \times \{-1, 0, 1\}^{\tau}$
corresponding to a symmetric strategy for rendezvous on $C_6$.  
We can give a lower bound on the probability this strategy fails: 
since there are only $3^\tau$ sequences of moves $\{-1,0,1\}$, the collision probability that the two players choose the same sequence is at least $3^{-\tau}=3^{-T/m}$.  In that case, unless they start at the same vertex of $C_6$, they do not meet.

It is conceivable that the players could meet in $C_n$ but not in $C_6$.  For instance, they could start at adjacent sectors of $C_n$ less than $m/2$ steps apart, meet after $t=m/2$ steps, and then return to their original positions.  
However, this can only occur if they are at adjacent vertices of $C_6$.  With probability $1/2$, their initial vertices on $C_6$ are $2$ or $3$ steps apart; as noted in the previous paragraph they follow the same path with probability at least $3^{-\tau}$, in which case their distance on $C_6$ stays the same.  Thus the probability that they fail to meet on $C_n$ is at least $(1/2) 3^{-\tau}$, completing the proof.
%
%
\end{proof}

\section{Open Questions and Directions for Further Work}


We have obtained a number of new results for the symmetric rendezvous problem.  In particular, we have shown that we can rendezvous with high probability in $T = 4 n$ steps, but not in $T = \alpha n$ steps for any $\alpha < 4$.  We have also partially addressed the question from~\cite{anderson1990rendezvous} regarding whether strategies of Anderson-Weber type are asymptotically optimal.  If the goal is to maximize the probability of meeting within $T$ steps, they are indeed optimal for $T \le n$, but not for $T \ge 4n$.

There are still many unanswered questions.   We close with some suggestions for further work.
\begin{itemize}
\item What happens for $n < T < 4n$?  Is the optimal strategy some kind of interpolation between Anderson-Weber and our rendezvous codes?
\item How quickly can the probability of success approach one for $T > 4n$?  If we simply repeat the $T=4n$ rendezvous code, we fail with probability $(\log n)^{-O(T/n)}$, reducing the probability of failure to $1/\poly(n)$ with $T=O(n \log n / \log \log n)$.  Can we do better?  In terms of rendezvous codes, this corresponds to the obvious product construction; but are there rendezvous codes of size $\poly(n)$ of length, say, $T=O(n)$?
\item We have seen that we can meet on some vertex-transitive graphs much more quickly than on the complete graph; namely, we can meet with high probability in $n$ steps on the decorated cycle.  Here the limited set of automorphisms (namely the dihedral group of order $2n$) helps us more than the restriction of moving along the edges hurts us.  However, it seems plausible that if $G$'s automorphism group is sufficiently rich, that our lower bound of $4n$ steps applies.  In particular, we conjecture that this is the case on the hypercube.
\end{itemize}
Finally, it is still possible, as conjectured in~\cite{anderson1990rendezvous}, that the Anderson-Weber strategy asymptotically minimizes the expected time to meet.

\section*{Acknowledgments}  

The authors would like to thank Richard Weber for offering helpful 
feedback on an early draft of this paper.
We also gratefully recognize two anonymous referees for numerous
corrections and suggestions for improvements.

T.H. and C.M. are supported by NSF grant CCF-1219117.
C.M. and A.R. are supported by NSF grant CCF-1117426.  
C.M. is also supported in part by the John Templeton Foundation.  
T.H. is also supported by NSF grant CCF-1150281.


\end{document}